\documentclass[12pt,a4paper]{amsart} 
\usepackage[utf8]{inputenc}
\usepackage[english]{babel}
\usepackage{amsmath}
\usepackage{amsfonts}
\usepackage{amssymb}
\usepackage{mathrsfs}
\usepackage{color}

\usepackage{amsthm}
\newtheorem{thm}{Theorem}

\theoremstyle{definition}

\newtheorem{prb}{Problem}

\usepackage{graphicx}
\usepackage{epstopdf}
\usepackage{epsfig}

\usepackage{url}
\usepackage{hyperref}

\usepackage[a4paper,top=2.5cm,bottom=2.8cm,
left=3.2cm,right=3.2cm]{geometry}

\newcounter{num} 
\newcommand{\Fg}[1][]{\thenum}

\newcommand\Al{\alpha}
\newcommand\B{\beta}

\newcommand\G{\gamma}

\newcommand\CA{\mathcal{A}}
\newcommand\CB{\mathcal{B}}

\definecolor{pink}{rgb}{1.0,.4,.0} 
\definecolor{green2}{rgb}{.7,.7,.7} 
\definecolor{pink2}{rgb}{1.,.8,.1} 
\definecolor{blue2}{rgb}{.6,.5,1} 
\definecolor{blue3}{rgb}{.4,.6,.2} 
\definecolor{pink3}{rgb}{.7,.4,.4} 
\definecolor{mag}{rgb}{.6,.0,1.0} 
\definecolor{rd}{rgb}{1.,.0,.0} 

\markleft{\hfill \textsc{A note on a problem invoving a square in a 
curvilinear triangle} \hfill}
\begin{document}
\bigskip

\begin{center}
{\Large \textbf{A note on a problem invoving a square in a 
curvilinear triangle}} \\
\medskip
\bigskip
\textsc{Hiroshi Okumura} \\


\end{center}
\bigskip

\textbf{Abstract.} A problem involving a square in the 
curvilinear triangle made by two touching congruent circles and their common 
tangent is generalized.



\bigskip\bigskip


Let $\Al_1$ and $\Al_2$ be touching circles of radius $a$ with external 
common tangent $t$. In this note we consider the following  problem  
\cite{ad, ertts, yjkgh} 
(see Figure \ref{fpr}). 

\begin{prb}\label{p1}
{\rm $ABCD$ is a square such that the side $DA$ lies on $t$ and the 
points $C$ and $B$ lie on $\Al_1$ and $\Al_2$, respectively. Show that 
$2a=5|AB|$. }
\end{prb}


\medskip
\begin{center}
\includegraphics[clip,width=95mm]{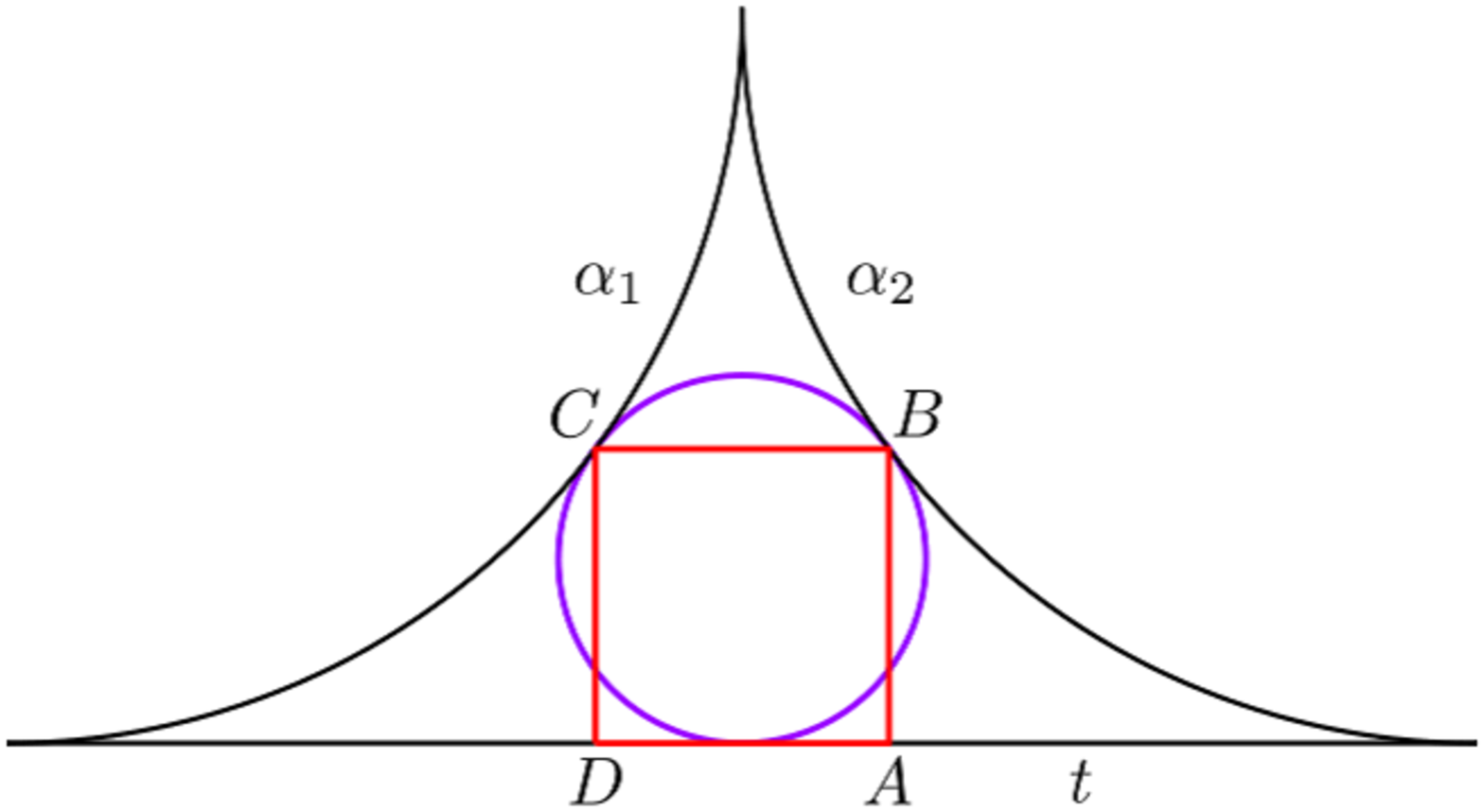}\refstepcounter{num}\label{fpr}\\
\medskip
Figure \Fg . 
\end{center}


If $\G_1$, $\G_2$, $\cdots$, $\G_n$ are congruent circles touching 
a line $s$ from the same side such that $\G_1$ and $\G_2$ touch 
and $\G_i$ $(i=3,4,\cdots, n)$ touches $\G_{i-1}$ from the side opposite 
to $\G_1$, then $\G_1$, $\G_2$, $\cdots$, $\G_n$ are called congruent 
circles on $s$. 
The curvilinear triangle made by $\Al_1$, $\Al_2$ and $t$ is denoted by 
$\Delta$. The incircle of $\Delta$ touches $\Al_1$ and $\Al_2$ at $C$ 
and $B$, respectively as in Figure \ref{fpr}. Indeed the problem is 
generalized as follows (see Figure \ref{f5}):

\begin{thm}\label{t1}
If $\B_1$, $\B_2$, $\cdots$, $\B_n$ are congruent circles on $t$ lying in 
$\Delta$ such that $\B_1$ touches $\Al_1$ at a point $C$ and $\B_n$ touches 
$\Al_2$ at a point $B$ and $A$ is the foot of perpendicular from $B$ to $t$, 
then the following relations hold. \\ 
\noindent{\rm (i)} $n|AB|=|BC|$. \\
\noindent{\rm (ii)} $2a=\left(\left(\sqrt{n}+1\right)^2+1\right)|AB|$. 
\end{thm}

\begin{proof}
Let $b$ be the radius of $\B_1$. By Theorem 5.1 in \cite{OKSJM171} we have 
\begin{equation}\label{eq1}
a=(\sqrt{n}+1)^2b. 
\end{equation} 
Let $d=|AB|$. Since $C$ divides the 
segment joining the centers of $\Al_1$ and $\B_1$ in the ratio $a:b$ 
internally, we have 
\begin{equation}\label{eq2}
\frac{d-b}{b}=\frac{a-b}{a+b}. 
\end{equation} 
Eliminating $b$ from \eqref{eq1} and \eqref{eq2}, and solving the resulting 
equation for $d$, we get $d=2a/(1+(1\pm\sqrt{n})^2)$. But in the minus sign 
case we get $2b-d=2a(1-4\sqrt{n})/(n^2-n+2\sqrt{n}+2)<0$ by \eqref{eq1}. 
Hence $d=2a/(1+(1+\sqrt{n})^2)$. This proves (ii). Let $|BC|=2h$. Then 
from the right triangle 
formed by the line $BC$, the segment joining the centers of $\Al_1$ and 
$\B_1$, and the perpendicular from the center of $\Al_1$ to $BC$, we get  
$(a-h)^2+(a-d)^2=a^2$. 
Solving the equation for $h$, we have $h=a-\sqrt{(2a-d)d}=an/(1+(1+\sqrt{n})^2)$. 
This proves (i). 
\end{proof}

\medskip
\begin{center}
\includegraphics[clip,width=115mm]{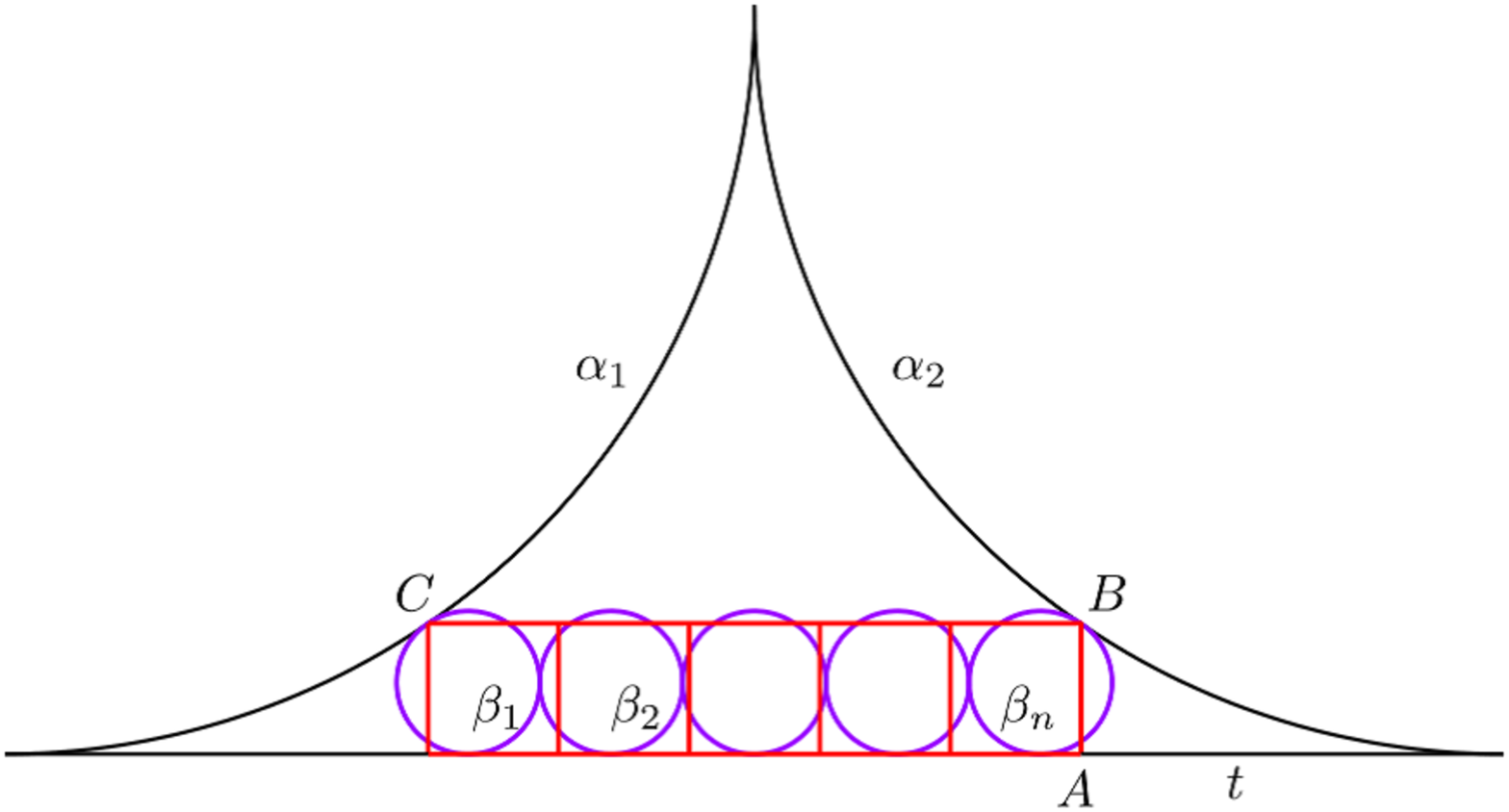}\refstepcounter{num}\label{f5}\\
Figure \Fg : $n=5$ 
\end{center}

The figure consisting of $\Al_1$, $\Al_2$, $\B_1$, $\B_2$, $\cdots$, $\B_n$ 
and $t$ is denoted by $\CB(n)$ and considered in \cite{OKSJM171}. Since the 
inradius of $\Delta$ equals $a/4$, the  next theorem is also a generalization 
of Problem \ref{p1} (see Figure \ref{fa4}). 

\begin{thm}\label{t2}
Let $\B_1$, $\B_2$, $\cdots$, $\B_n$ be congruent circles on a line $s$. 
If a circle $\Al$ of radius $a$ touches $s$ and $\B_1$ and $\B_n$ externally 
at points $C$ and $B$, respectively, $A$ is the foot of 
perpendicular from $B$ to $s$, then the following relations hold. \\ 
\noindent{\rm (i)} $(n-1)|AB|=|BC|$. \\
\noindent{\rm (ii)} $2a=((n-1)^2+4)|AB|/4$. 
\end{thm}

\medskip
\begin{center}
\includegraphics[clip,width=105mm]{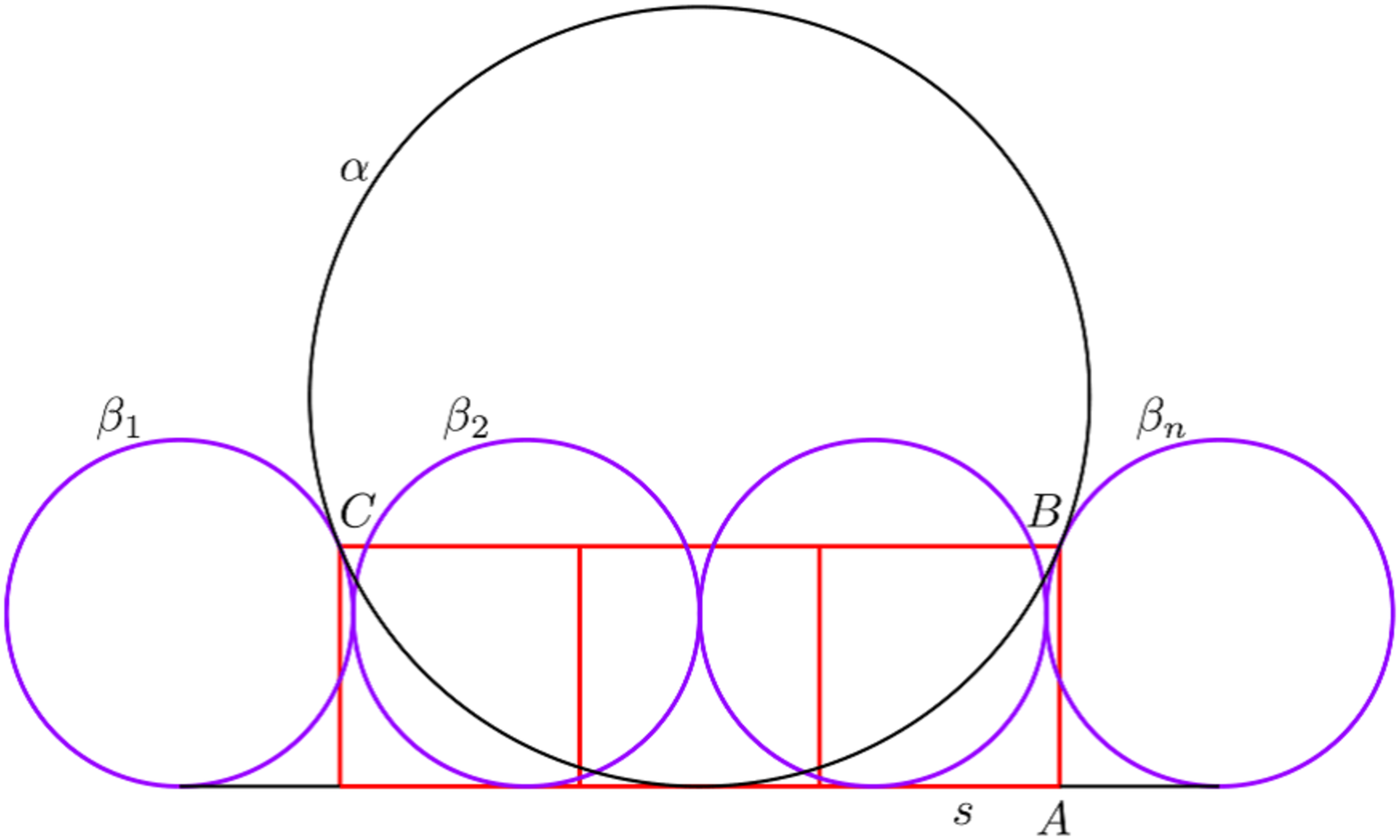}\refstepcounter{num}\label{fa4}\\
Figure \Fg : $n=4$ 
\end{center}

Theorem \ref{t2} is proved in a similar way as Theorem \ref{t1} using the 
fact that the ratio of the radii of $\Al$ and $\B_1$ equals $(n-1)^2:4$ 
\cite{OKMIQ}. 
The figure consisting of $\Al$, $\B_1$, $\B_2$, $\cdots$, $\B_n$ 
and $s$ is denoted by $\CA(n)$ and considered in \cite{OKSJM171}.

\end{document}